\documentclass[a4paper]{llncs}
\usepackage{amsmath}
\usepackage{amsfonts}
\usepackage{amssymb}
\usepackage{graphics}
\usepackage{epsfig}
\usepackage{graphicx}
\usepackage{epstopdf}
\usepackage[all]{xy}
\usepackage{color,soul}
\newcommand{\cl}[1]{\operatorname{cl}_R(#1)}
\newcommand{\F}{\mathbb F}
\newcommand{\Z}{\mathbb Z}
\newcommand{\Perf}{\operatorname{Perf}}

\title{On the non-existence of perfect codes in the NRT-metric}
\author{Viviana Gubitosi, Aldo Portela \and Claudio Qureshi}
\institute{Instituto de Matem\'atica y estad{\'\i}stica Rafael Laguarda, Universidad de la Rep{\'u}blica, Montevideo, 11300, Uruguay.\\
\email{gubitosi@fing.edu.uy, aldo@fing.edu.uy, cqureshi@fing.edu.uy}}

\begin{document}

\maketitle

\begin{abstract}
In this paper we consider codes in $\F_q^{s\times r}$ with packing radius $R$ regarding the NRT-metric (i.e. when the underlying poset is a disjoint union of $s$ chains with the same length $r$) and we establish necessary condition on the parameters $s,r$ and $R$ for the existence of perfect codes. More explicitly, for $r,s\geq 2$ and $R\geq 1$ we prove that if there is a non-trivial perfect code then $(r+1)(R+1)\leq rs$. We also explore a connection to the knapsack problem and establish a correspondence between perfect codes with $r>R$ and those with $r=R$. Using this correspondence we prove the non-existence of non-trivial perfect codes also for $s=R+2$.
\end{abstract}

\section{Introduction}

Most of the research in error-correcting codes deal with the Hamming metric which can efficiently approach the communication problems arising from channels where the channel noise generates equiprobable errors. However, when possible errors form patterns of specific shape, Hamming metric is not the appropriate one to suit the characteristics of the channels. Several metrics have been introduced to deal with these possible patterns of errors. In this way, in 1997, Rosenbloom and Tsfasman \cite{RT97} introduced a metric on linear spaces over finite fields, motivated by applications to interference in parallel channels of communication systems. This metric was previously used by Niederreiter \cite{Niederreiter87} related to the study of sequences with low discrepancy which play an important role in quasi-Monte Carlo methods and other application in numerical analysis. Nowadays this metric is known as the Niederreiter-Rosenbloom-Tsfasman metric (or NRT metric for short). Several central concepts on codes in Hamming spaces have been investigated in NRT spaces, such as perfect codes, MDS codes, weight distribution, self-dual NRT-codes, packing and covering problems, see for instance \cite{BGL95,CM15,CMS19,Niederreiter87,RT97,SA20,YSBY10}. Some of these concepts have been investigated also in the context of block codes in NRT spaces \cite{PFA10}. Construction of codes such as Reed-Solomon and BCH codes have also generalized to NRT-spaces \cite{ZLA16}. On the one hand, the NRT metric is a special case of poset metric \cite{BGL95} (corresponding to the case when the poset is a disjoint union of chains of the same length) and on the other hand the NRT metric generalizes the Hamming metric (the latter corresponds to the special case when the poset is an antichain). \\

This paper deals with the existence of perfect codes in NRT spaces. An NRT space is a metric space $(\F_{q}^{s\times r},d)$, where $\F_{q}^{s\times r}$ is the set of $s\times r$ matrices over the finite field $\F_q$ (where $q$ is the number of elements of the finite field) and $d$ is the NRT metric associated with this space. This metric coincides with the poset metric associated with a poset consisting of $s$ disjoint chains of length $r$. A code $C$ is any non-empty subset of $\F_{q}^{s\times r}$ and its elements are called codewords. Let $R>0$. We say that $C\subseteq \F_{q}^{s\times r}$ is a perfect code (or more precisely, an $R$-perfect code) if the balls of radius $R$ centered at codewords are disjoint and their union is $\F_{q}^{s\times r}$. A perfect code is called trivial when $|C|=1$ or $C=\F_{q}^{s\times r}$. The problem of determining for which parameters $(s,r,R)$, there is a (non-trivial) $R$-perfect code $C\subseteq \F_{q}^{s\times r}$ is a difficult task and there are few results in this direction. The case $r=1$ (corresponding to the Hamming metric) was settled by  Tiet{\"a}v{\"a}inen in \cite{Tietavainen73} .

\begin{theorem}[\cite{Tietavainen73}]
Every non-trivial perfect code $C\subseteq \F_{q}^{s\times 1}$, regarding the Hamming metric, have the same parameters that a repetition code, a Hamming code or a Golay code. Therefore, we have the following four possibilities:
\begin{itemize}
\item[i)] (Repetition code) $C\subseteq \F_{2}^{s\times 1}$ with $s\equiv 1 \pmod{2}$, $|C|=2$ and $R=1$;
\item[ii)] (Hamming code) $C\subseteq \F_{q}^{s\times 1}$ with $s=\frac{q^i-1}{q-1}$, $|C|=q^{s-i}$ and $R=1$;
\item[iii)] (Binary Golay code) $C\subseteq \F_{2}^{23\times 1}$ with $|C|=2^{11}$ and $R=3$ or
\item[iv)] (Ternary Golay code) $C\subseteq \F_{3}^{11\times 1}$ with $|C|=3^6$ and $R=5$.
\end{itemize}
\end{theorem}

The case $s=1$ has been done in the paper of Brualdi et al. \cite{BGL95} (see also \cite{FAPP18}). They show that if $R\leq r$ and $f:\F_{q}^{1\times (r-R)}\to \F_{q}^{1\times R}$ is any function then $C_{f}=\{(f(y),y): y \in \F_q^{1\times r}\}$ is an $R$-perfect code in $\F_{q}^{1\times r}$ and every perfect code comes from this construction. Since the cases $r=1$ and $s=1$ are well established, we can restrict to the case $r,s\geq 2$.

As far as we know, there are only some few results about the existence (or non-existence) of perfect codes in the NRT metric. For example, it is known that there are no (non-trivial) perfect codes for $s=2$ chains \cite{BGL95}. Constructions of new perfect codes from old ones are given in \cite[Chapter~4.3.1]{FAPP18}. The main result of this paper is a necessary condition for the existence of perfect codes that involves the three fundamental parameter $s,r$ and $R$ in a non-trivial way. We introduce the quantity $\delta = (r+1)(R+1)-sr-1$ and prove that there are no non-trivial perfect codes with parameter $(s,r,R)$ provided that $\delta\geq 0$. We also provide a new construction of perfect codes from old ones that allows to extended the non-existence result to the case $s=R+2$.\\ 

This paper is organized as follows. In Section \ref{Sec:Notation} we review some definitions and notation used throughout the paper. In Section \ref{Sec:Sticky} we introduce the notion of $R$-sticky vectors and $R$-decomposable vectors that play an important role in the proof of our results. We consider the problem of determining when two $R$-balls are disjoint in a NRT space and prove that this is equivalent to a particular instance of the knapsack problem. Section \ref{Sec:Main} is dedicated to the proof of our main result which bring a necessary condition for the existence of perfect codes in NRT spaces (Theorems \ref{Th:deltageq1} and \ref{Th:delta0}). In section \ref{Section:lifting} we present an elementary construction of new perfect codes from old one (Proposition \ref{prop:general_lifting}) and show how to use to prove the non-existence of perfect codes in the case $s=R+2$ (Proposition \ref{prop:nonexistenceFromLifting}). In the last section we conclude with some further remarks.

\section{Notation and definitions}\label{Sec:Notation}

As usual, for a positive integer $s$, we denote $[s]:=\{i\in \Z^{+}: 1\leq i \leq s\}$. For subsets $I,J \subseteq [s]$, $[s]=I \uplus J$ means that $[s]=I \cup J$ and $I \cap J = \emptyset$. If $(X,d)$ is a metric space, $x\in X$ and $R\geq 0$ we denote by $B(x,R):=\{y\in X: d(x,y)\leq R\}$ the (closed) $R$-ball centered at $x$. If in addition $X$ has a vector space structure and the metric is translation-invariant, we denote the ball of radius $R$ centered at the zero vector $0$ by $B(R):=B(0,R)$. It is clear that $B(x,R)=x+B(R)=\{x+b:b\in B(R)\}$ for every $x\in X$. In general, for $B,C \subseteq X$ and $x\in X$, we denote by $B+x:=\{b+x: b \in B\}$ and $B+C:=\{b+c: b\in B, c\in C\}$. The equality $X=B\oplus C$ means that every $x\in X$ can be written univocally as $x=b+c$ with $b\in B$ and $c\in C$ (or equivalently, $X=B+C$ and the sets $B+c$ with $c\in C$ are disjoint).\\

Let $q$ be a prime power. We denote by $\F_{q}$ the finite field with $q$ elements and by $\F_{q}^{s\times r}$ the set of $s\times r$ matrices over $\F_{q}$. In this paper we identify (in the obvious way) the matrix space $\F_{q}^{1\times r}$, the set of $r$-tuples $\F_{q}^{r}$ and the function set $\operatorname{Func}([r],\F_{q}):=\{x:[r]\to \F_q\}$.\\

The NRT-weight function $w:\F_{q}^{s\times r}\to [0,+\infty)$ is defined as follows. For $s=1$ and $x=(x_1,\ldots, x_r)\in \F_{q}^{r}$ we have $$w(x):=  \left\{ \begin{array}{ll} \max\{i\in [r]: x_i\neq 0\} & \textrm{ \ \ if } x\neq 0; \\ 0 & \textrm{ \ \ if } x= 0. \end{array} \right.$$

For $s\geq 2$, if $x= \left( \begin{array}{c} x_1 \\ \vdots \\ x_s \end{array} \right) \in \F_q^{s\times r}$ where $x_i \in \F_q^{1\times r}$ is the $i$-th row of $x$, the NRT-weight is extended additively as $w(x):=\sum_{i=1}^{s}w(x_i)$.\\

The NRT-metric $d:\F_{q}^{s\times r} \to [0,+\infty)$ is the metric induced by the NRT-weight, i.e. $d(x,y):=w(x-y)$, for all $x,y$ in $\F_q^{s\times r}$. It is clear that this metric is translation-invariant. Note that if $r=1$ and $x,y$ are in  $\F_q^{s\times 1}$ then the NRT-weight $w(x)=\#\{i\in [s]: x_i \neq 0\}$ equals the Hamming weight of $x$ and the NRT-metric $d(x,y)=\#\{i \in [s]: x_i \neq y_i\}$ equals the Hamming distance between $x$ and $y$. \\

Let $C \subseteq \F_{q}^{s\times r}$ be a code. The covering radius of $C$ is the minimum integer $R\geq 0$ such that $\F_{q}^{s\times r}$ is the union of the $R$-balls centered at codewords (i.e. $\F_{q}^{s\times r}= B(R)+C$); in this case we say that $C$ is an $R$-covering. Covering codes under the NRT metric is a very active area of research with many connections with other areas of mathematics and also with application to telecommunication \cite{CM15,Quistorff07}. The packing radius of $C$ is the maximum integer $R'\geq 0$ such that the $R'$-balls centered at codewords are disjoint; in this case we say that $C$ is an $R'$-packing. It is clear that $R'\leq R$. When $R'=R$, we say that $C$ is a perfect code (or $R$-perfect code). To avoid trivial cases we assume $|C|>1$ and $C \neq \F_q^{s\times r}$. Since the translation maps are isometries, we can assume without loss of generality that $0\in C$. The set of all (non-trivial) $R$-perfect codes in $\F_{q}^{s\times r}$ is denoted by $\operatorname{Perf}(s,r,R)$. Associated with this set we introduce the quantity $\delta:= (r+1)(R+1)-sr-1$ which play an important role in our non-existence results. \\

As mentioned in the introduction, the NRT codes can be seen as a special case of poset codes (i.e. codes in spaces endowed with a poset metric). Poset codes were introduced in the seminal paper of Brualdi et al. \cite{BGL95} and there is an extensive literature on this topic (see, for instance, the book of Firer  et al. \cite{FAPP18} and the references therein). Let $(P,\preceq)$ be a finite partially ordered set (poset). A chain is a subset $S\subset P$ such that any two elements of $S$ are comparable. A subset
$I \subseteq P$ is an ideal of $P$ if $b \in I$ and $a \preceq  b$, implies $a \in I$. The ideal generated by a subset $A$ of $P$ is the ideal of smallest cardinality that contains $A$, denoted by $\langle A\rangle$. Every finite poset $P$ induces a distance (called the $P$-distance) in the function space $\F_q^{P}:=\{x:P \to \F_q\}$ given by $d_{P}(x,y)=|\langle i \in P: x(i)\neq y(i) \rangle|$. The $P$-weight of $x \in \F_q^{P}$ is defined by $w_P(x)=d_P(x,0)$, where $0$ denotes the zero function. Clearly we have the relation $w_P(x-y)=d_P(x,y)$. Let $s,r \in \Z^{+}$;  when $P=[s]$ we identify the set $\F_q^{P}$ with $\F_q^s$ and when $P=[s]\times [r]$ we identify the set $\F_q^{P}$ with $\F_q^{s\times r}$ in the natural way. An NRT poset is a poset with underlying set $P=[s]\times [r]$ and the order given by $(i,j)\preceq (i',j')$ if $i=i'$ and $j\leq j'$. In this case $P$ is the disjoint union of the chains $\{i\}\times [r]$ for $i=1,\ldots,s$ and the $P$-metric regarding this poset is just the NRT metric defined above. In the spacial case that $P$ is an antichain (i.e. there are no two comparable elements) the corresponding $P$-metric is the Hamming metric.

\section{$R$-sticky and $R$-decomposable vectors}\label{Sec:Sticky}

In this section we prove some preliminary results. We start by introducing the notions of $R$-sticky and $R$-decomposable vectors which play an important role in the proof of our main theorem. We also show a connection between the problem of determining when a vector is $R$-decomposable with the knapsack problem.

\subsection{$R$-sticky vectors}

In a broader context we can consider a finite metric space $(X,d)$ and a real number $R>0$. We denote by $B(c,R)=\{x\in X: d(x,c)\leq R\}$, the closed ball centered at $c$.\\

We say that a subset $S\subseteq X$ is $R$-open if $S$ can be written as a (non necessarily disjoint) union of balls of radius $R$ (or $R$-balls for short). An $R$-closed set is defined as the complement of an $R$-open set. We define the $R$-closure operator as $\cl{S}= \{p \in X: \textrm{every $R$-ball containing $p$ also contains some point of $S$}\}.$\\

It is clear by definition that $S\subseteq \cl{S}$ and the equality holds if and only if $S$ is $R$-closed.

\begin{definition}
Let $S\subseteq X$. The elements of $\cl{S}\setminus S$ are called $R$-sticky vectors for $S$.
\end{definition}

Now we consider a group $(X,+)$ embedded with an invariant-by-translation metric $d$ and a subset $C\subseteq X$ whose elements are called {\it codewords}. It is clear that the $R$-ball centered at the origin $B(R)$ satisfies $B(c,R)=c+B(R)$ for every $c\in X$. We say that a set $C\subseteq X$ is an $R$-perfect code (or a perfect code of radius $R$) if $X = C \oplus B(R)$, i.e. the $R$-balls centered at codewords are disjoint and they cover $X$. To avoid trivial cases we also assume here that $|C|>1$ and $C\neq X$. The following lemma has immediate verification.

\begin{lemma}\label{lemma:R-Perfect-Implies-RClosedBall}
If there is an $R$-perfect code in $X$ then the ball $B(R)$ is $R$-closed.
\end{lemma}

In our case of interest we consider the space $X=\F_{q}^{s\times r}$ provided with the NRT-metric. Our proof of non-existence of perfect codes for the case $\delta>0$ consists in proving that the ball $B(R)$ is not $R$-closed. 

\subsection{$R$-decomposable vectors}

This subsection deals with the problem of determining when two $R$-balls intercept in an NRT-space. In the Hamming case (i.e. $r=1$) we have an easy criterion: $B(x,R)\cap B(x',R)=\emptyset$ if and only if $w(x-x')\geq 2R+1$. This property does not hold in general for NRT-spaces. For example, if $x \in \F_q^{3\times 2}$ is the vector whose three rows equals to $e_2=(0,1)$ and $x'$ equals the zero vector, then $B(x,3)\cap B(3)=\emptyset$ but $w(x-x')=6<2\cdot 3+1$. We prove in this and the next subsections that the problem of determining when two $R$-balls intercept in an NRT-space is equivalent to a special instance of the knapsack problem.\\

The next definition give us a useful criterion to determine when two $R$-balls intersect. We start  introducing the following notation. If $x= \left( \begin{array}{c} x_1 \\ \vdots \\ x_s \end{array} \right) \in \F_q^{s\times r}$ and $I \subseteq [s]$, we denote by $w(x|_I):= \sum_{i\in I}w(x_i)$ and by $x_I \in \F_q^{s\times r}$ the vector obtaining from $x$ substituting each row $x_j$ with $j\not\in I$ by the null vector of $\F_q^{1\times r}$. It is clear that $w(x_I)=w(x|_I)$.

\begin{definition}
Let $R\geq 1$ and $x \in \F_q^{s\times r}$. An $(x,R)$-partition of $[s]$ is a pair $(I,J)$ of subsets of $[s]$ such that $[s]=I \uplus J$, $w(x|_I)\leq R$ and $w(x|_J)\leq R$. If $[s]$ admits an $(x,R)$-partition we say that $x$ is $R$-decomposable. Otherwise, we say that $x$ is $R$-indecomposable.
\end{definition}

Note that a vector $x \in \F_q^{s\times r}$ such that $w(x)\leq R$ is always $R$-decomposable (for example we can take $I=\emptyset$ and $J=[s]$).\\          

The following lemma  establishes the ultrametric property for the NRT-metric in the case $s=1$ and it has a direct verification.

\begin{lemma}\label{lemma:ultrametric}
Let $x,y \in \F_q^{1\times r}$. Then $w(x+y) \leq  \max\{w(x), w(y)\}$. Moreover, if $w(x)\neq w(y)$, then the equality holds.
\end{lemma}

The ultrametric property will be used in the proof of several results as well as the next lemma.

\begin{lemma}\label{Lemma:Decomposable}
Let $x, x'\in \F_q^{s\times r}$. We have that $B(x,R)\cap B(x',R)\neq \emptyset$ if and only if $x-x'$ is $R$-decomposable.
\end{lemma}

\begin{proof}
Since the NRT-metric $d$ is translation invariant we can assume without loss of generality that $x'=0$. If $b\in B(x,R)\cap B(R)$ then $w(b)\leq R$ and $w(x-b)\leq R$. We consider the sets $I=\{i\in [s]: w(b_i)\geq w(x_i-b_i)\}$ and $J=[s]\setminus I$. By Lemma \ref{lemma:ultrametric} we have $w(x|_I)=\sum_{i\in I}w\left((x_i-b_i)+b_i\right)\leq \sum_{i\in I}w(b_i)\leq w(b)\leq R$ and $w(x|_J)=\sum_{j\in J}w((x_j-b_j)+b_j)=\sum_{j\in J}w(x_j-b_j)\leq w(x-b)\leq R$.  Therefore,  $x$ is $R$-decomposable. Conversely, if $x$ is $R$-decomposable and $[s]=I \uplus J$ with $w(x|_I)\leq R$ and $w(x|_J)\leq R$ then $x-x_I=x_J \in B(x,R)\cap B(R)\neq \emptyset$. \qed
\end{proof}

\subsection{ Relation with the knapsack problem}

Now we briefly discuss about the computational problem of determining if a given vector $x \in \F_q^{s\times r}$ is $R$-decomposable. Next we show that this problem can be reduced to a particular instance of the knapsack problem. Given positive real numbers $v_1,\ldots, v_s$ (called values); $w_1,\ldots, w_n$ (called weights) and $W$ (called the weight capacity of the knapsack); the problem of finding a subset of index $I\subseteq [s]$ which maximize $\sum_{i\in I} v_i$ restricted to the condition $\sum_{i\in I}w_i\leq W$ is known as the knapsack problem \footnote{The version of the knapsack problem we are considering here is sometimes called the $0-1$ knapsack problem.}. 

\begin{proposition}
Let $R\geq 1$, $x= \left( \begin{array}{c} x_1 \\ \vdots \\ x_s \end{array} \right) \in \F_q^{s\times r}$ and $w_i:= w(x_i)$ for $1\leq i \leq s$. Let $I$ be a solution of the knapsack problem with values $w_1,\ldots, w_s$; weights $w_1,\ldots, w_s$ and weight capacity of the knapsack $R$. Denote $J:= [s]\setminus I$. We have that $x$ is $R$-decomposable if and only if $\sum_{j\in J}w(x_j)\leq R$.
\end{proposition}

\begin{proof}
Assume that $x$ is an $R$-decomposable vector and let $(I_0,J_0)$ be an $(x,R)$-partition of $[s]$. Since $\sum_{i\in I_0}w(x_i)\leq R$ and $I$ is a solution of the knapsack problem, we have $\sum_{i\in I_0}w(x_i)\leq \sum_{i\in I}w(x_i)$. Therefore $$\sum_{j\in J}w(x_j) = w(x) - \sum_{i\in I}w(x_i) \leq w(x) - \sum_{i\in I_0}w(x_i) = \sum_{j\in J_0}w(x_j)\leq R. $$ 

To prove the converse we assume now that $\sum_{j\in J}w(x_j)\leq R$ where $I=[s]\setminus J$ is a solution of the knapsack problem with values and weights given as above. Since $I$ is a solution of the knapsack problem we also have $\sum_{i\in I}w(x_i)\leq R$. Thus, $(I,J)$ is an $(x,R)$-partition of $[s]$ and $x$ is $R$-decomposable. \qed
\end{proof}

\section{Non-existence of perfect codes for $\delta \geq 0$}\label{Sec:Main}

In this section we prove our main result, the non-existence of perfect codes for $\delta\geq 0$. The general strategy is to construct a special point $m=m(R)\in \F_{q}^{s\times r}$ with the property that it is an $R$-sticky vector for $B(R)$ whenever $\delta=r(R+1-s)+R>0$ (in particular, this implies the non-existence of perfect codes for this case). Otherwise, if $\delta \leq 0$, we prove that there are points $p\in \F_q^{s\times r}$ satisfying $p \in B(m,R)$ and $B(p,R)\cap B(R)=\emptyset$ (i.e. $m$ is not an $R$-sticky vector). However, such points $p$ have to verify several conditions. Using these conditions, for the case $\delta=0$ we can obtain that we will call an $R$-sticky set for $B(R)$. We use such a set to extend the non-existence result also for this case.\\

Clearly, if $B(R)=\F_{q}^{s\times r}$ then $\operatorname{Perf}(s,r,R)=\emptyset$. For this reason, we always assume that $B(R)\subsetneq \F_{q}^{s\times r}$ (i.e. there is a vector $x\in \F_{q}^{s\times r}$ with $w(x)=R+1$).

We introduce a new parameter $t:=s-R-1$ and write $\left\{\begin{array}{l}
s=R+1+t; \\ R=tr+\delta.
\end{array} \right.$ \\

The following theorem deals with the case $\delta= r(R+1-s)+R\geq 1$. 

\begin{theorem}\label{Th:deltageq1}
Let $s,r,R$ be positive integers with $s\geq 2$. If $\delta \geq 1$, there is an $R$-sticky vector $m \in \F_q^{s\times r}$ for $B(R)$. In particular, there are no $R$-perfect codes in $\F_q^{s\times r}$ regarding the NRT-metric for $\delta\geq 1$.
\end{theorem}

\begin{proof}
Denote by $\{e_1,\ldots, e_r\}$ the canonical basis of $\F_{q}^{1\times r}$ and  by $e_0$ the null vector of $\F_{q}^{1\times r}$. Let $\ell$ and $h$ be the integers such that $R+1=(\ell+1)s-h$ with $\ell\geq 0$ and $1\leq h \leq s$. Note that $\ell$ is the unique integer such that $\ell s \leq R+1 < (\ell+1)s$. 

We define $m= \left( \begin{array}{c} m_1 \\ \vdots \\ m_s \end{array} \right) \in \F_q^{s\times r}$ where $m_i=\left\{ \begin{array}{cl}
e_{\ell+1} \hspace{.5cm} & \textrm{if \ \ } 1\leq i \leq s-h; \\
e_{\ell} \hspace{.5cm} &   \textrm{if \ \ } s-h+1\leq i \leq s.
\end{array} \right.$

Clearly, $w(m)=(\ell+1)s-h=R+1$. Consider any vector $c\in \mathbb{F}_{q}^{s\times r}$ such that $m\in B(c,R)$. In order to prove that $m$ is an $R$-sticky vector for $B(R)$ we have to show that the intersection $B(c,R)\cap B(R)\neq \emptyset$. By Lemma \ref{Lemma:Decomposable}, it is equivalent to prove that $c$ is $R$-decomposable. The proof will be divided into two cases.\\

Case $t\leq 0$: We have that $R+1\geq s$ and $\ell\geq 1$. If $c=m$, since $s\geq 2$, it is easy to check that the pair $([s-1],\{s\})$ is a $(c,R)$-partition of $[s]$ and therefore $c$ is $R$-decomposable. Otherwise, if $c\neq m$ we consider the sets $I:=\{i\in [s]: w(c_i)=w(m_i)\}$ and $J:=[s]\setminus I$. We assert that $(I,J)$ is a $(c,R)$-partition of $[s]$. Indeed, $w(c|_{I})=\sum_{i\in I}w(m_i)=R+1-\sum_{j\in J}w(m_j)\leq R+1-\ell\cdot |J|\leq R$  and by Lemma \ref{lemma:ultrametric}, we have $w(c|_{J})= \sum_{j\in J}w(c_j) \leq \sum_{j\in J} w(c_j-m_j)\leq d(c,m)\leq R$.\\

Case $t\geq 1$:  We have that $\ell=0$, $R+1=s-h$, $m_i=e_1$ for $1\leq i \leq R+1$ and $m_i=e_0$ otherwise. Consider $I=\{i\in [R+1]: c_i=e_1\}$ and $J=[s]\setminus I$. If $|I|=R+1$ we assert that $(I',J')$ with $I'=[R]$ and $J'=[s]\setminus [R]$ is a $(c,R)$-partition of $[s]$. Indeed, $w(c|_{I'})$ $=\sum_{i=1}^R w(c_i)=R$ and $w(c|_{J'})= w(e_1)+ \sum_{i=R+2}^s w(c_i)\leq 1+(s-R-1)r\leq \delta+tr=R$. Then, it only remains to verify the case  $|I|\leq R$. In this case we assert that $(I,J)$ is a $(c,R)$-partition of $[s]$. Indeed, $w(c|_{I})=|I|\leq R$ and by Lemma \ref{lemma:ultrametric}, $w(c_j)\leq w(c_j-e_1)$ for every $j\in J \cap [R+1]$, then $w(c|_{J}) = \sum_{j \in J\cap [R+1]}w(c_j)+\sum_{j=R+2}^{s}w(c_j) \leq \sum_{j \in J\cap [R+1]}w(c_j-e_1)+\sum_{j=R+2}^{s}w(c_j)= w(c-m)\leq R$.\\

The existence of an $R$-sticky vector for $B(R)$ implies that the ball $B(R)$ is not $R$-closed with respect to the NRT-metric and by Lemma \ref{lemma:R-Perfect-Implies-RClosedBall} there are no $R$-perfect codes in $\F_{q}^{s\times r}$. \qed
\end{proof}

Now we focus on the case $\delta\leq 0$ which corresponds to $R\leq tr$. In particular, $t\geq 1$ and $s=R+1+t\geq R+2$. \\

We introduce the following notation for $x\in \F_q^{s\times r}$. Since $s=(R+1)+t$, we can write   $x=\left( \begin{array}{c} x^{+} \\ x^{-} \end{array} \right)$ where $x^{+}\in \F_q^{(R+1)\times r}$ and $x^{-}\in \F_q^{t\times r}$. It is clear that $w(x)=w(x^-)+w(x^+)$.\\

Note that our proof of the non-existence of $R$-perfect codes for the case $\delta\geq 1$ was based in the fact that the ball $B(R)$ is not $R$-closed but this is not longer true for $\delta\leq 0$. In fact, we have the following proposition. 

\begin{proposition}
If $\delta\leq 0$ then the ball $B(R)$ is $R$-closed.
\end{proposition}

\begin{proof}
We have to prove that for each $x=\left( \begin{array}{c} x^{+} \\ x^{-} \end{array} \right)  \in \F_{q}^{s\times r}$ with $w(x)\geq R+1$ there is a vector $c=\left( \begin{array}{c} c^{+} \\ c^{-} \end{array} \right) \in \F_q^{s\times r}$ such that $w(c-x)\leq R$ and $B(c,R)\cap B(R)=\emptyset$. In the case that $w(x)\geq 2R+1$ we can take $c=x$. Now suppose that $w(x)\leq 2R$. Since permutation of rows are isometries of NRT-spaces, we can assume without loss of generality that $w(x_i)\geq w(x_j)$ for $1\leq i < j \leq s$ which implies $w(x^+)\geq R+1$ and $w(x^-)\leq R-1$. Consider the set $S=\{y \in \F_q^{t\times r}: w(y)\leq R,\ w(y_i)\geq w(x_i^{-}) \textrm{ for } 1\leq i \leq t\}$. Clearly, $S\neq \emptyset$ (because $x^{-} \in S$). Let $c \in \F_{q}^{s\times r}$ be a vector with $c^{+}=x^{+}$ and $c^{-}$ is any element with maximal weight in $S$. We assert that $w(c^{-})=R$. Indeed, if $w(c^{-})<R$, since $tr\geq R$ there is some row $c_{i}^{-}$ of $c^{-}$ with weight $w(c_{i}^{-})=\ell$ for some $\ell<r$. If $c' \in \F_{q}^{t\times r}$ is the vector obtained from $c^{-}$ by substituting their $i$-th row $c_{i}^{-}$ by $e_{\ell +1}+c_{i}^{-}$. We have that $w(c')=w(c^{-})+1\leq R$ and $c'\in S$ which contradicts the fact that $c^{-}$ has maximal weight. By Lemma \ref{lemma:ultrametric}, we have $w(c-x)=\sum_{i=1}^{t}w(c_{i}^{-}-x_{i}^{-})\leq \sum_{i=1}^{t}w(c_{i}^{-})=w(c^{-})=R$ and $w(c)=w(c^+)+w(c^{-})\geq R+1+R=2R+1$ which implies $B(c,R)\cap B(R)=\emptyset$. \qed
\end{proof}

We note that the fact that $B(R)$ is $R$-closed is equivalent to the non-existence of $R$-sticky vectors for $B(R)$. Next we extend the definition of $R$-sticky vectors.

\begin{definition}
Let $X \subseteq \F_{q}^{s\times r}\setminus B(R)$. We say that $X$ is an $R$-sticky set for $B(R)$ if for every cover of $X$ by disjoint $R$-balls, some of the balls intersect $B(R)$. That is, if $c_1,\ldots, c_k \in \F_{q}^{s\times r}$ are such that $B(c_i,R)\cap B(c_j,R)=\emptyset$ for $1\leq i < j \leq k$ and $X \subseteq \bigcup_{i=1}^{k}B(c_i,R)$ then there is an index $i$ such that $B(c_i,R)\cap B(R)\neq \emptyset$.
\end{definition}

To prove the non-existence of perfect codes for $\delta=0$ we need two lemmas, the first of them is an extension of Lemma \ref{lemma:R-Perfect-Implies-RClosedBall} and has direct verification.

\begin{lemma}\label{lemma:R-perfectImpliesNoStickySet}
If there is an $R$-perfect code in $\mathbb{F}_{q}^{s\times r}$ then there are no $R$-sticky sets for $B(R)$.
\end{lemma}

\begin{lemma}\label{lemma:Restriction-For-m}
Let $s,r$ and $R$ be positive integers with $s\geq 2$ and assume that $\delta=R-rt\leq 0$. Consider the vector $m \in \F_q^{s\times r}$ such that each row of $m^{+}\in \F_q^{(R+1)\times r}$ equals the canonical vector $e_1=(1,0, \cdots, 0)\in\F_{q}^{1\times r}$ and each row of $m^{-}\in \F_q^{t\times r}$ is the zero vector $0\in \mathbb{F}_{q}^{1\times r}$. Let $c\in \F_q^{s\times r}$ be a vector satisfying $c \in B(m,R)$ and $B(c,R)\cap B(R)=\emptyset$. Then, $c^{+}=m^{+}$ and $w(c^{-})=R$.
\end{lemma}

\begin{proof}
Consider $I=\{i\in [R+1]: c_i=e_1\}$ and $J=\{i\in [R+1]: c_i\neq e_1\}$. Take $J'=J\cup \{R+2, \cdots, s \}$. We have that $I\biguplus J'=[s]$. If $j\in J$, Lemma \ref{lemma:ultrametric} implies that $w(c_j)\leq w(c_j-e_1)$. Then, we have $w(c|_{J'})=\sum_{j\in J'}w(c_j)= \sum_{j\in J}w(c_j)+\sum_{i=R+2}^{s}w(c_i)\leq \sum_{j\in J}w(c_j-e_1)+\sum_{i=R+2}^{s}w(c_i)=d(c,m)\leq R$. By Lemma \ref{Lemma:Decomposable}, $c$ is $R$-indecomposable, thus  $w(c|_{I})=\sum_{i\in I} w(c_i)=|I|>R$. Then, $I=[R+1]$, $c^{+}=m^{+}$ and $w(c^{-})=w(c^{-}-m^{-})=d(c,m)\leq R$. It only remains to prove the inequality $w(c^{-})\geq R$. Again, since $c$ is $R$-indecomposable and $\sum_{i=1}^{R}w(c_i)=\sum_{i=1}^{R} w(e_1)= R$  we conclude that $\sum_{i=R+1}^{s}w(c_i)\geq R+1$.  Therefore, $w(c^{-})=\sum_{i=R+2}^{s}w(c_i)\geq R+1-w(c_{R+1})=R$. \qed
\end{proof}

\begin{remark}\label{Remark:isometry}
To each bijection $\theta:[s]\to [s]$ we can associate a map $\widehat{\theta}:\F_q^{s\times r}\to \F_q^{s\times r}$ such that $\widehat{\theta}(x)_{i}=x_{\theta^{-1}(i)}$ for $1\leq i \leq s$ (i.e. $\widehat{\theta}(e_i)=e_{\theta(i)}$, where $e_i$ is the $i$-th canonical vector). These maps are linear isometries of $\F_q^{s\times r}$ regarding the NRT-metric (because they act as permutation of rows). Consider $m\in \F_{q}^{s\times r}$ as in Lemma \ref{lemma:Restriction-For-m}. If $c \in \F_{q}^{s\times r}$ is such that $\widehat{\theta}(m) \in B(c,R)$ and $B(c,R)\cap B(R)=\emptyset$ then $\widehat{\theta}^{-1}(c)$ is in the hypothesis of Lemma \ref{lemma:Restriction-For-m} and then $c_{\theta(i)}=e_1$ for $1\leq i \leq R+1$ and $\sum_{i=R+2}^{s}w(c_{\theta(i)})=R$. In the next theorem we use this fact with $\theta$ being a cyclic shift.
\end{remark}

\begin{theorem}\label{Th:delta0}
Let $s,r,R$ be positive integers with $s,r\geq 2$. If $\delta = 0$, there is an $R$-sticky set with two elements for $B(R)$. In particular, there are no $R$-perfect codes in $\F_q^{s\times r}$ regarding the NRT-metric for $\delta=0$.
\end{theorem}

\begin{proof}
We consider $m\in \F_{q}^{s\times r}$ as in Lemma \ref{lemma:Restriction-For-m} ; i.e. $m_i=e_1$ for $1\leq i \leq R+1$ and $m_i=0$ for $R+2\leq i \leq s$. Let $\widehat{\theta}:\F_q^{s\times r}\to \F_q^{s\times r}$ be  the cyclic shift map induced by the permutation $\theta(i)=i+1$ for $1\leq i <s$ and $\theta(s)=1$. We assert that $S=\{m, m'\}$ with $m'=\widehat{\theta}(m)$ is an $R$-sticky set for $B(R)$. Indeed, consider $c, c' \in \F_{q}^{s\times r}$ such that $m\in B(c,R)$, $m'\in B(c',R)$, $B(c,R)\cap B(R)=\emptyset$ and $B(c',R)\cap B(R)=\emptyset$. It suffices to prove that $c\neq c'$ and $B(c,R)\cap B(c',R)\neq \emptyset$. 

By Lemma \ref{lemma:Restriction-For-m} and Remark \ref{Remark:isometry} we have that $c_{i}=c'_{i+1}=e_1$ for all $1\leq i \leq R+1$, $w(c^-)=\sum_{i=R+2}^{s}w(c_i)=R $ and $ w(c'_1)+\sum_{i=R+3}^{s}w(c'_i) = R$. Therefore  
$$w(c_{R+2}) =  R - \sum_{i=R+3}^{s}w(c_i) \geq R - (s-R-2)r = R-(t-1)r= \delta +r \geq 2.$$
Analogously $w(c'_1)\geq 2$. Since $w(c_{R+2})\geq 2$ and $c'_{R+2}=e_1$ it is clear that $c\neq c'$. It only remains to show that $c-c'$ is $R$-decomposable. Consider the sets $I'=\{R+3\leq i \leq s : w(c_i)\geq w(c'_i)\}$, $J'=\{ R+3\leq i \leq s : w(c_i)<w(c'_i)\}$, $I=\{R+2\}\bigcup I'$ and $J=\{1,2,\ldots,R+1\}\bigcup J'$.  Clearly $[s]=I\biguplus J$. Using Lemma \ref{lemma:ultrametric} together with the inequality $w(c_{R+2})\geq 2$ we obtain:
\begin{eqnarray*}
w((c-c')|_{I}) & =  & w(c_{R+2}-e_1)+\sum_{i\in I'} w(c_i-c'_i)  \leq    w(c_{R+2})+\sum_{i\in I'} w(c_i)\\
& \leq &  \sum_{i=R+2}^s w(c_i)  =  R
\end{eqnarray*}
Analogously, again by Lemma \ref{lemma:ultrametric} together with the inequality $w(c'_1)\geq 2$ we obtain
\begin{eqnarray*}
w((c-c')|_{J}) & =  & w(e_1-c'_1)+\sum_{i\in J'}w(c_i-c'_i)  = w(c'_1)+\sum_{i\in J'}w(c'_i)  \\
& \leq &  w(c'_1)+ \sum_{i=R+3}^s w(c'_i)  =  R
\end{eqnarray*}
Thus, $c-c'$ is $R$-decomposable and consequently $B(c,R)\cap B(c',R)\neq \emptyset$. \qed
\end{proof}

The next corollary is a direct consequence of Theorems \ref{Th:deltageq1} and \ref{Th:delta0}.

\begin{corollary}
Let $s,r,R$ be positive integers with $s\geq 2$. If there is an $R$-perfect code in $\F_q^{s\times r}$ regarding the NRT-metric then $(R+1)(r+1)\leq sr$. 
\end{corollary}

\section{A lifting result and the non-existence of perfect codes for the case $s=R+2$}\label{Section:lifting}

There are several constructions in the literature of new perfect codes from old ones. In this section we identify (in the obvious way) the space $\F_q^{s\times (r+h)}$ with $\F_q^{s\times r} \times \F_q^{s\times h}$. For convenience, we write any point $x=(x',x'')\in \F_q^{s\times (r+h)}$ with $x'\in \F_q^{s\times r}$ and $x''\in \F_q^{s\times h}$ (note that $w(x)\leq r$ implies $x''=0$). A general construction of perfect codes is given by Firer et.al. in \cite[Chapter~4]{FAPP18}. This construction extends some other previous constructions and can be stated as follows.

\begin{proposition}[\cite{FAPP18}]\label{Prop:trivial_lifting}
Let $R,r$ and $h$ be positive integers with $R \leq r$. Let $C'$ be an $R$-perfect code in $\F_q^{s\times r}$. Then, $C:=C'\times \F_{q}^{s\times h}$ is an $R$-perfect code in $\F_{q}^{s\times (r+h)}$.
\end{proposition}

We denote by $\Perf(s,r,R)$ the set of non-trivial perfect codes in $\F_q^{s\times r}$ of radius $R$. A direct consequence of Proposition \ref{Prop:trivial_lifting} is that if $\Perf(s,R,R)\neq \emptyset$ then $\Perf(s,r,R)\neq \emptyset$ for every $r>R$.\\

The converse in the Proposition \ref{Prop:trivial_lifting} does not hold in general, that is, there are $R$-perfect codes $C\subseteq \F_q^{s\times (r+h)}$ which cannot be obtained from a single perfect code $C\in \F_q^{s\times r}$. In order to extend the non-existence results given in Theorems \ref{Th:deltageq1} and \ref{Th:delta0} for other values of $\delta<0$ we need a more general construction. 

\begin{proposition}\label{prop:general_lifting}
Let $R,r$ and $h$ be positive integers with $R \leq r$. Assume that $\Perf(s,r,R)\neq \emptyset$. Consider a function $f: \F_{q}^{s\times h}\to \Perf(s,r,R)$ and let $C_{f}:=\{(c',c''): c'' \in \F_{q}^{s\times h}, c'\in f(c'')\}\subseteq \F_q^{s\times (r+h)}$. Then $C_{f}\in \Perf(s,r+h,R)$. Conversely, every $R$-perfect code in $\F_{q}^{s\times (r+h)}$ can be constructed from $R$-perfect codes in $\F_{q}^{s\times r}$ in this way.
\end{proposition}

\begin{proof}
Since $R\leq r$, note that if $x,y \in \F_q^{s\times (r+h)}$ verify $w(x-y)\leq R$, then $x'' = y''$. Now, consider a code $C_{f}\subseteq \F_q^{s\times (r+h)}$ for some function $f: \F_{q}^{s\times h}\to \Perf(s,r,R)$. First we prove that $C_{f}$ is an $R$-covering. Let $x=(x',x'') \in \F_q^{s\times (r+h)}$. Define $c''=x''$. Since $f(c'')$ is an $R$-perfect code in $\F_{q}^{s\times r}$ and $x'\in \F_{q}^{s\times r}$, there is a point $c'\in f(c'')$ such that $w(x'-c')\leq R$. Then, $c=(c',c'')\in C_{f}$ satisfies $w(x-c)=w(x'-c')\leq R$ so $C_{f}$ is an $R$-covering. Now, we prove that $C_{f}$ is an $R$-packing. Let $c_1=(c_1',c_1'')$ and $c_2=(c_2',c_2'')$ be two codewords in $C_{f}$ such that $B(c_1,R) \cap B(c_2,R)\neq \emptyset$. Consider $x=(x',x'')\in \F_q^{s\times (r+h)}$ such that $w(x-c_1)\leq R$ and $w(x-c_2)\leq R$. By the initial observation we have that $c_1''=x''=c_2''$ and then $x' \in B(c_1',R)\cap B(c_2',R)$, where $c_1'$ and $c_2'$ are codewords of the perfect code $f(x'')$. This is possible only if $c_1'=c_2'$, which implies $c_1=c_2$. This prove that $C_{f}$ is an $R$-packing and we conclude that $C_{f}$ is an $R$-perfect code. To prove the converse we consider a perfect code $C\in \Perf(s,r+h,R)$ and define the function\footnote{As usual, $2^S$ denotes the power set of $S$.} $f:\F_{q}^{s\times h}\to 2^{\F_{q}^{s\times r}}$ given by $f(c'')=\{c' \in \F_q^{s\times h}: (c',c'')\in C\}$. We assert that $f(c'')\in \Perf(s,r,R)$, for every $c''\in \F_{q}^{s\times h}$. Indeed, if $x' \in \F_{q}^{s\times r}$ we consider the point $x=(x',c'')$ and the codeword $c_1=(c_1',c_1'')\in C$ such that $w(x-c_1)\leq R$. This implies that $c''=c_1''$ and $w(x'-c_1')=w(x-c_1)\leq R$ with $c_1' \in f(c_1'')=f(c'')$. Thus, $f(c'')$ is an $R$-covering of $\F_{q}^{s\times r}$. To prove that $f(c'')$ is an $R$-packing consider two codewords $c_1',c_2' \in f(c'')$ such that $w(c_1'-c_2')\leq R$. By definition $c_1:=(c_1',c'')$ and $c_2:=(c_2',c'')$ belong to the $R$-perfect code $C$ and $w(c_1-c_2)=w(c_1'-c_2')\leq R$. This is possible only if $c_1=c_2$, which implies $c_1'=c_2'$. This prove that $f(c'')$ is an $R$-packing.
\end{proof}

\begin{corollary}\label{coro:general_lifting}
The following assertions are equivalent:
\begin{itemize}
\item[i)] $\Perf(s,R,R)\neq \emptyset$;
\item[ii)] $\Perf(s,r,R)\neq \emptyset$ for some $r>R$;
\item[iii)] $\Perf(s,r,R)\neq \emptyset$ for every $r>R$.
\end{itemize}
\end{corollary}

Now we can extend our non-existence results for some negative value of $\delta$.

\begin{proposition}\label{prop:nonexistenceFromLifting}
If $s=R+2$ and $R\geq 2$ then $\Perf(s,r,R)=\emptyset$.
\end{proposition} 

\begin{proof}
By contradiction, assume there is an $R$-perfect code $C\subseteq \F_{q}^{s\times r}$. By Theorems \ref{Th:deltageq1} and \ref{Th:delta0} we have $\delta_C:=r(R-s+1)+R=R-r<0$, thus $R<r$. By Proposition \ref{prop:general_lifting} this implies the existence of a perfect code $C'\in \Perf(s,R,R)$ with $\delta_{C'}=R(R-s+1)+R=R-R=0$ which contradicts Theorem \ref{Th:delta0}.
\end{proof}

\begin{corollary}
For every $h\in \Z^+$, there are values of $s,r$ and $R$ such that $\delta=-h$ and $\Perf(s,r,R)=\emptyset$.
\end{corollary}

\begin{proof}
Apply Proposition \ref{prop:nonexistenceFromLifting} with $s=r+2+h$, $R=r+h$ and $r\geq 1$.
\end{proof}

\section{Further remarks}\label{Sec:FurtherRemarks}

In this paper we consider codes in the NRT-metric for $s\geq 2$ chains of length $r\geq 2$. We prove the non-existence of $R$-perfect codes if $\delta:=(r+1)(R+1)-sr-1\geq 0$ (Theorems \ref{Th:deltageq1} and \ref{Th:delta0}). Using a lifting construction (Proposition \ref{prop:general_lifting}) we were able to extended the non-existence results for parameters $(s,r,R)$ such that $s=R+2$. The construction given in Proposition \ref{prop:general_lifting} reduces the problem of studying the existence (or non-existence) of perfect codes to the case $R\geq r$. We conjecture that in this case, the only (non-trivial) perfect codes are the perfect Hamming codes mentioned in the introduction.


\begin{thebibliography}{99}

\bibitem{BGL95}
R.A. Brualdi, J.S. Graves, K.M. Lawrence. Codes with a poset metric.
Discrete Math., 147(1-3), 57-72 (1995).

\bibitem{CM15}
A.G. Castoldi, E.L. Monte Carmelo. The covering problem in
Rosenbloom-Tsfasman spaces. Electron. J. Combin., 22(3), paper 3.30
(2015).

\bibitem{CMS19}
A.G. Castoldi, E.L. Monte Carmelo, R. da Silva. Partial sums of binomials,
intersecting numbers, and the excess bound in Rosenbloom-Tsfasman
space. Comput. Appl. Math., 38(55) (2019).

\bibitem{FAPP18}
M. Firer, M. M. S. Alves, J. A. Pinheiro and L. Panek. Poset codes: partial orders, metrics and coding theory. Springer International Publishing (2018).

\bibitem{Niederreiter87}
H. Niederreiter. Point sets and sequences with small discrepancy.
Monatsh. Math., 104(4), 273-377 (1987).

\bibitem{PFA10}
L. Panek, M. Firer and M. M. S. Alves. Classification of Niederreiter–Rosenbloom–Tsfasman block codes. IEEE Transactions on information theory, 56(10), 5207-5216. 2010.

\bibitem{Quistorff07} J. Quistorff. On Rosenbloom and Tsfasman’s generalization of the Hamming space. Discrete Math., 307:2514–2524 (2007).

\bibitem{RT97}
M. Y. Rosenbloom and M. A. E. Tsfasman. Codes for the m-metric. Problemy Peredachi Informatsii, 33(1), 55-63 (1997).

\bibitem{SA20}
Santos, W. and Alves, M. M. Polynomial invariant theory and shape enumerator of self-dual codes in the NRT-metric. IEEE Transactions on Information Theory, 66(7), 4061-4074. 2020.

\bibitem{Tietavainen73}
A. Tiet{\"a}v{\"a}inen. On the nonexistence of perfect codes over finite fields. SIAM Journal on Applied Mathematics, 24(1), 88-96 (1973).

\bibitem{YSBY10}
B.Yildiz, I. Siap, T. Bilgin, G. Yesilot. The covering problem for finite
rings with respect to the RT-metric. Appl. Math. Lett., 23(9), 988-992
(2010).

 \bibitem{ZLA16}
W. Zhou, S. Lin and K. A. Abdel-Ghaffar. BCH Codes for the Rosenbloom–Tsfasman Metric. IEEE Transactions on Information Theory, 62(12), 6757-6767. 2016.

\end{thebibliography}
\end{document}